\documentclass[10pt, twoside, reqno]{amsart}

\usepackage[english]{babel}

\usepackage{amsmath,amssymb,amsthm}

\usepackage{enumitem}

\newcommand{\si}{\sigma}

\newcommand{\ga}{\gamma}

\newcommand{\un}[1]{{\underline{#1}} }

\newcommand{\N}{{\mathbb N}}

\newcommand{\R}{{\mathbb R}}
\newcommand{\C}{{\mathbb C}}
\newcommand{\K}{{\mathbb K}}

\newcommand{\Sym}{{\mathcal S}}

\newcommand{\expan}[2]{#1^{[#2]}}

\DeclareMathOperator{\parti}{Part}
\DeclareMathOperator{\charakt}{char}
\DeclareMathOperator{\id}{id}

\DeclareMathOperator{\tr}{tr}

\newcommand{\upto}{,\ldots ,}

\newtheorem{thm}{Theorem}

\newtheorem{lem}{Lemma}[section]

\newtheorem{cond}[lem]{Conditions}

\theoremstyle{definition}

\newtheorem{remark}[lem]{Remark}

\newenvironment{proof_of}[1]{\medskip\noindent{\it Proof #1}}
{\hfill$\Box$ \bigskip}

\numberwithin{equation}{section}

\usepackage[pdfpagemode=None,
           	pdftitle={Separating invariants for multisymmetric polynomials},
           	pdfauthor={Artem Lopatin, Fabian Reimers}]{hyperref}

\begin{document}

\title{Separating invariants for multisymmetric polynomials}

\author{Artem Lopatin, Fabian Reimers}

\address{State University of Campinas, 651 Sergio Buarque de Holanda, 13083-859 Campinas, SP, Brazil}

\email{artem\underline{\;\;}lopatin@yahoo.com}

\address{Technische Universit\"at M\"unchen, Zentrum Mathematik - M11, 
Boltzmannstr.~3, 85748 Garching, Germany}

\email{reimers@ma.tum.de}

%\date{November 1, 2019}



\thanks{Acknowledgments. The first author was supported by FAPESP 2019/10821-8. We are grateful for this support.}

\begin{abstract}
This article studies separating invariants for the ring of multisymmetric polynomials in $m$ sets of $n$ variables over an arbitrary field $\K$. We prove that in order to obtain separating sets it is enough to consider polynomials that depend only on $\lfloor \frac{n}{2} \rfloor + 1$ sets of these variables. This improves a general result by Domokos about separating invariants. In addition, for $n \leq 4$ we explicitly give minimal separating sets (with respect to inclusion) for all $m$ in case $\charakt(\K) = 0$ or $\charakt(\K) > n$.
\end{abstract}

\maketitle

%%%%%%%%%%%%%%%%%%%%%%%%%%%%%%%%%%%%%%%%%%%%%%%%%%%%%%%%%%%%%%%%%%%%%%
\section{Introduction}
%%%%%%%%%%%%%%%%%%%%%%%%%%%%%%%%%%%%%%%%%%%%%%%%%%%%%%%%%%%%%%%%%%%%%%

Multisymmetric polynomials are the generalization of symmetric polynomials in $n$ variables $x_1 \upto x_n$ to $m \geq 2$ sets of $n$ variables. They have been classically studied in characteristic 0 (see Schl\"afli~\cite{schlaefli1852}, Junker~\cite{junker1893}, Weyl~\cite{weyl1939classical}, and for a historical overview see Domokos~\cite[Remark~2.6]{domokos2007}). In the setting of invariant theory symmetric polynomials appear as the elements of the invariant ring $\K[V]^{\Sym_n}$ of the standard representation $V = \K^n$ of the symmetric group $\Sym_n$, where $\K$ is a field. Multisymmetric polynomials form the invariant ring $\K[V^m]^{\Sym_n}$, where the action of $\Sym_n$ on $V$ is extended to a diagonal action of $\Sym_n$ on $V^m = V \oplus \ldots \oplus V$. The coordinate ring of $V^m$ is denoted by
\[ \K[V^m] = \K[ \,x(j)_{i} \mid i = 1 \upto n,\, j = 1 \upto m],\]
where $x(j)_{i}:V^m\to \K$ sends $(a_1,\ldots,a_m) \in V^m$ to the $i$-th coordinate $(a_j)_i$ of $a_j$.

\noindent
An obvious invariant inside $\K[V^m]$ is obtained for any $m$-tuple $\underline{k} = (k_1 \upto k_m) \in \N^m$, where $\N$ is the set of all non-negative integers, and any $1 \leq t \leq n$ by 
\begin{equation}\label{eqDefinitionOfF}  \si_t(\un{k})=\sum\limits_{1\leq i_1<\ldots< i_t\leq n} x^{\un{k}}_{i_1} \cdot \ldots \cdot x^{\un{k}}_{i_t} \quad \in \K[V^m]^{\Sym_n},\end{equation}
where 
\[ x^{\un{k}}_i=x(1)_i^{k_1}\cdot \ldots \cdot x(m)_i^{k_m} \quad \text{ (for $1\leq i\leq n$).}\] The invariants $\si_t(\un{k})$ are called \emph{elementary multisymmetric polynomials}.
For short, we write $\tr(\un{k})=\si_1(\un{k})$ for the power sum
\[\tr(\un{k}) = \si_1(\un{k})= \sum\limits_{i=1}^n x^{\un{k}}_i.\]
There is a natural $\N^m$-grading on $\K[V^m]$, which is preserved by the $\Sym_n$-action. Under this grading $\si_t({\underline{k}})$ is multi-homogeneous of multi-degree $t\underline{k}$. 

Upper bounds on the degrees of elements of minimal generating sets for the ring of multisymmetric polynomials $\K[V^m]^{\Sym_n}$ were studied by Fleisch\-mann~\cite{fleischmann}, Vaccarino~\cite{vaccarino2005}, Domokos~\cite{domokos2009vector}, and a minimal generating set was explicitly described by Rydh~\cite{rydh2007}. In particular,  $\K[V^m]^{\Sym_n}$ is known to be generated by the $m$ elements $\si_n(0,\ldots,0,1,0,\ldots,0)$, where 1 is at position $j$ for $1 \leq j \leq m$, together with the elements $\si_t(\un{k})$ with $1\leq t\leq n$ and $\un{k} \in \N^m$ satisfying $\gcd(\un{k})=1$ and $k_1<\frac{n}{t}$,$\ldots$,  $k_m<\frac{n}{t}$ 
(see Corollary 5.3 of~\cite{domokos2009vector}). In case $\charakt(\K) =  0$ or $\charakt(\K)> n$ the ring $\K[V^m]^{\Sym_n}$ is minimally (w.r.t. inclusion) generated as a $\K$-algebra by 
\begin{equation}\label{eqMinimalGen} 
\tr(\un{k})\;\; \text{ with } \quad  \un{k} \in \N^m \text{ satisfying } 1 \leq k_1 + \ldots + k_m \leq n,
\end{equation} 
see Domokos~\cite[Theorem 2.5]{domokos2009vector}.

%To a set of $m$-tuples $I \subset \N^m$ we assign the set of multi-homogeneous invariants
%\begin{equation}\label{eqDefinitionOfS}  S(I) := \{ f_{k_1 \upto k_m} \mid (k_1 \upto k_m) \in I \} .\end{equation}
%Throughout this paper we assume that . 

While the set (\ref{eqMinimalGen}) of generating invariants can not be improved, it is often helpful to consider a set of separating invariants, which can be significantly smaller than (\ref{eqMinimalGen}).
In 2002 Derksen and Kemper~\cite{derksen2002computational} introduced the notion of separating invariants as a weaker concept than generating invariants. Given a subset $S$ of an invariant ring $\K[W]^{G}$, we say that elements $u,\, v \in W$ {\it can be separated by $S$} if there exists an invariant $f\in S$ with $f(u)\neq f(v)$. The set $S$ is called {\it separating} if any $u,\, v \in W$ that can be separated by $\K[W]^G$ can also be separated by $S$. It is a well-known fact about the invariants of a finite group $G$ that $u$, $v$ can be separated by $\K[W]^G$ if and only if they are not in the same orbit (see~\cite[Section 2.4]{derksen2002computationalv2}).

Since the introduction of the notion of separating invariants, many results were found where they behave better than generating invariants. Some properties of separating invariants can be found in the second volume of the book by Derksen and Kemper~\cite[Section 2.4]{derksen2002computationalv2}. Explicit separating sets and upper degree bounds of separating sets have been calculated for some group actions, see e.g. \cite{derksen2018}, \cite{domokos2017}, \cite{draisma2008polarization}, \cite{dufresne2014}, \cite{dufresne2015separating}, \cite{elmer2014}, \cite{elmer2016}, \cite{kaygorodov2018}, \cite{kohls2013}, \cite{reimers2018}.

%For a positive integer $m_0 \leq m$ and a set $I \subset \N^{m_0}$ of $m_0$-tuples we define its {\it expansion}  $\expan{I}{m} \subset \K[V^m]^{\Sym_n}$ to a set of $m$-tuples by filling in zeros at all possible places. As an example, for $I = \{ (2,1) \}$ we have $\expan{I}{3} = \{ (2,1,0),\,(2,0,1),\,(0,2,1) \}.$
%The formal definition is 
%\begin{align}\label{expansion} \expan{I}{m} :=  \big\{ &(k_1 \upto k_m) \in \N^m \mid \, \exists \, (j_1 \upto j_{m_0}) \in I \text{ and } \exists \, (i_1 \upto i_{m_0}) \in \N^{m_0} \notag\\  &\text{such that } 1 \leq i_1 < i_2 < \ldots < i_{m_0} \leq m, \text{ } \notag  \,\, k_{i_l} = j_l \text{ for } l = 1 \upto m_0, \notag \\ &\text{and } k_l = 0 \text{ for all } l \notin \{ i_1 \upto i_{m_0}\} \,  \big\}. 
%\end{align}
%Given a set $S \subset \K[V^{m_0}]^{\Sym_n}$, its {\it expansion} $\expan{S}{m} \subset \K[V^m]^{\Sym_n}$ is defined as the set of all elements
%$$\si_t(\un{k}), \;\; \text{ where }\un{k}\in \expan{\{\un{j}\}}{m} \text{ for some } \si_t(\un{j})\in S.$$

In this paper we are interested in multi-homogeneous separating sets of $\K[V^m]^{\Sym_n}$. One way to obtain them is by \emph{expanding} a separating set for $\K[V^{m_0}]^{\Sym_n}$ for some $m_0 \leq m$. By this we mean the following construction. We say that an $m_0$-tuple $\un{j}\in\N^{m_0}$ is {\it $m$-admissible} if $1\leq j_1<\cdots < j_{m_0}\leq m$. For any $m$-admissible $\un{j}\in\N^{m_0}$ and $f\in \K[V^{m_0}]^{\Sym_n}$ we define the invariant $f^{(\un{j})}\in \K[V^{m}]^{\Sym_n}$ as the result of the following substitution of variables in $f$: 
\[ x(1)_i \to x(j_1)_i,\, \ldots, \, x(m_0)_i \to x(j_{m_0})_i \quad \text{ (for all } 1 \leq i \leq n).\]
Given a set $S \subset \K[V^{m_0}]^{\Sym_n}$, we define its {\it expansion} $\expan{S}{m} \subset \K[V^m]^{\Sym_n}$ by 
\begin{equation}\label{DefExpansion}
\expan{S}{m} = \{ f^{(\un{j})} \mid f \in S \text{ and } \un{j}\in\N^{m_0} \text{ is $m$-admissible}\}.
\end{equation}
As an example, for $S = \{ \si_t(2,1) \}\subset \K[V^2]^{\Sym_n}$ for some $t$ we have $$\expan{S}{3} = \{ \si_t(2,1,0),\,\si_t(2,0,1),\,\si_t(0,2,1) \}\subset \K[V^3]^{\Sym_n}.$$

The idea of studying separating sets of vector invariants that depend only on a smaller number of variables was considered by Domokos~\cite{domokos2007} for an $n$-dimensional representation $W$ of an arbitrary algebraic group $G$. Domokos showed that it is enough to consider a separating set of $\K[W^{m_0}]^G$ for $m_0 = 2 n$ and expand it to $\K[W^m]^G$ for $m \geq m_0$. This was later improved to $m_0 = n+1$ by Domokos and Szab\'o \cite{domokos2011helly}.
 
Our main theorem is a strengthening of this result in the case of multisymmetric polynomials to $m_0 = \lfloor \frac{n}{2} \rfloor + 1$, where $\lfloor \frac{n}{2} \rfloor$ denotes the largest integer $\leq \frac{n}{2}$.

\medskip
%Theorem 1---------------------------------------------------------------------
\begin{thm}\label{TheoremMain}  Assume that $m_0 \geq \lfloor \frac{n}{2} \rfloor + 1$ and that $S \subset \K[V^{m_0}]^{\Sym_n}$ is separating.  
Then for all $m \geq m_0$ the expansion $\expan{S}{m}$ as defined in (\ref{DefExpansion}) is separating for $\K[V^{m}]^{\Sym_n}$.
\end{thm}
\medskip

By Remark~\ref{remark_new} below it is enough to prove Theorem~\ref{TheoremMain} for a particular separating  set $S$. 

\medskip
%Remark 1.1---------------------------------------------------------------------
\begin{remark}(cf. \cite[Remark~1.3]{domokos2007})\label{remark_new}
Assume that $S_{1}$ and $S_{2}$ are separating sets for $\K[V^{m_0}]^{\Sym_n}$ and assume that $m>m_0$. Then $\expan{S_1}{m}$ is separating for $\K[V^{m}]^{\Sym_n}$ if and only if $\expan{S_2}{m}$ is separating for $\K[V^{m}]^{\Sym_n}$.
\end{remark}

Let $\sigma(n)$ denote the minimal number $m_0$ such that the expansion of some separating set $S$ for $\K[V^{m_0}]^{\Sym_n}$ produces a separating set for $\K[V^m]^{\Sym_n}$ for all $m \geq m_0$. By Remark~\ref{remark_new} this is independent of $S$ and
Theorem~\ref{TheoremMain} can be rephrased as 
\begin{equation}\label{eqSigma} \sigma(n) \leq \lfloor \frac{n}{2} \rfloor + 1.\end{equation}
In Theorem~\ref{TheoremMainMin} in Section \ref{MinimalSeparating} we explicitly give minimal (w.r.t. inclusion) separating sets for $n = 2,\,3,\,4$ in case $\charakt(\K) = 0$ or $\charakt(\K) > n$. These sets also show that the upper bound (\ref{eqSigma}) is exact for $n \leq 4$. For $n = 3$ and $\K = \C$ the algebra $\K[V^m]^{\Sym_3}$ was studied in details by Domokos and Pusk\'as \cite{domokos2012}.

This paper is structured as follows. At first we need a few basic results about partitions of the set $\{ 1 \upto n\}$, which are provided in Section~\ref{SectionPartitions}. Then we study the transitions from $n-1$ to $n$ and from $m-1$ and to $m$ in Section~\ref{SectionLemmasForInduction}. The lemmas from Section~\ref{SectionLemmasForInduction} will be applied in the proofs of Theorems~\ref{TheoremMain} and~\ref{TheoremMainMin} in Sections~\ref{SectionProofOfTheorem1}~and~\ref{MinimalSeparating}. In Remark~\ref{RemarkFraction} we compare the asymptotic sizes of our multi-homogeneous separating set and the minimal generating set (\ref{eqMinimalGen}) in case $\charakt(\K) = 0$ or $\charakt(\K) > n$.
%Most of this paper is devoted to the proof of Theorem~\ref{TheoremMain}. First need a few basic results about partitions of the set $\{ 1 \upto n\}$ which are provided in Section~\ref{SectionPartitions}. Then we study the transition between $n-1$ and $n$ and $m-1$ and $m$ in Section~\ref{SectionNotation}~and~\ref{SectionLemmasForInduction}. The  lemmas in Section~\ref{SectionLemmasForInduction} can be seen as partial cases of the proof of Theorem~\ref{TheoremMain}, but the point of these lemmas is that we can also use them for explicit studies of minimal separating sets for $n \leq 4$ in Section~\ref{MinimalSeparating}. In Section~\ref{SectionSize} we compare the asymptotic sizes of the multi-homogeneous separating sets and the minimal genarating set.

\textbf{Acknowledgment.} We are very grateful to the anonymous referee for his or her helpful comments, in particular, for pointing us towards a characteristic-free version of the results. 

%%%%%%%%%%%%%%%%%%%%%%%%%%%%%%%%%%%%%%%%%%%%%%%%%%%%%%%%%%%%%%%%%%%%%%
\section{Partitions of the set $[n]$}\label{SectionPartitions}
%%%%%%%%%%%%%%%%%%%%%%%%%%%%%%%%%%%%%%%%%%%%%%%%%%%%%%%%%%%%%%%%%%%%%%

Recall that a {\it partition} $A=\{I_1 \upto I_r\}$ of $[n] := \{ 1 \upto n\}$ is a set of non-empty subsets of $[n]$ with pairwise empty intersections such that $I_1\cup\ldots\cup I_r=[n]$. For a partition $A$ of $[n]$ denote by $G_A$ the subgroup of $G := \Sym_n$ that fixes the sets of $A$, i.e., 
\[ G_A := \{ \sigma \in \Sym_n \mid \sigma(I) = I \text{ for all } I \in A \}.\]
The intersection partition of two partitions $A$ and $B$ of $[n]$  is 
\[ A \sqcap B := \{ I \cap J \mid I \in A,\, J \in B \} \setminus \{ \emptyset \}.\]
Note that $\sqcap$ is an associative operation on the set of all partitions of  $[n]$.  For the corresponding subgroups we have
\begin{equation} \label{eq1}
G_{A \sqcap B} = G_A \cap G_B.    
\end{equation}
A vector $a = \begin{pmatrix}a_1 \\ \vdots \\ a_n \end{pmatrix} \in V$ defines an equivalence relation on $[n]$ through 
\[ i \sim j \quad \Longleftrightarrow \quad a_i = a_j .\]
This equivalence relation defines a partition, denoted by $\parti(a)$, of the set $[n]$. In particular, $|\parti(a)|$ is the number of distinct coordinates $a_1,\ldots,a_n$ of $a$. 
The stabilizer of $a$ under the $\Sym_n$-action on $V$ is equal to the subgroup $G_{\parti(a)}$ defined above: 
\begin{align}\label{StabilizerOfVector} G_a &:= \{ \sigma \in \Sym_n \mid a_{\sigma(i)} = a_i \text{ for } i = 1 \upto n \} \\ &\,\,= \{ \sigma \in \Sym_n \mid \sigma(I) = I  \text{ for all } I \in \parti(a) \} = G_{\parti(a)} .\notag \end{align}

\medskip
%Lemma L2.1---------------------------------------------------------------------
\begin{lem}\label{Lemma1OnPartitions} Let $A$ and $B$ be partitions of $[n]$. If $G_A \not\subset G_B$, then $|A \sqcap B| > |A|$.
\end{lem}
\begin{proof} The statement obviously follows from the fact that $G_{A}\subset G_{B}$ holds if and only if the partition $A$ is  a refinement of the partition $B$. 
\end{proof}
\medskip

Next for a partition $A$ of $[n]$ we define
\[ \min(A) := \min \{ |I| \mid I \in A \} .\]
Clearly, if $|A| \geq \lfloor \frac{n}{2} \rfloor + 1$, then $\min(A) = 1$. This fact together with an iterated use of Lemma~\ref{Lemma1OnPartitions} implies the next statement.
%\begin{equation}\label{eqMinimumPartition} |A| \geq \lfloor \frac{n}{2} \rfloor + 1 \quad \Longrightarrow \quad \min(A) = 1 .\end{equation}

\medskip
%Lemma L2.2---------------------------------------------------------------------
\begin{lem}\label{Lemma2OnPartitions} Let $A_1 \upto A_r$ be partitions of $[n]$ with $r \geq \lfloor \frac n 2 \rfloor$ such that for all $i = 1 \upto r$ we have
\[ G_{A_1} \cap \ldots \cap G_{A_{i-1}} \not\subset G_{A_i}, \]
where for $i = 1$ the empty intersection is interpreted as $\Sym_n$. Then 
$$\min(A_1 \sqcap \ldots \sqcap A_r) = 1.$$
\end{lem}

%%%%%%%%%%%%%%%%%%%%%%%%%%%%%%%%%%%%%%%%%%%%%%%%%%%%%%%%%%%%%%%%%%%%%%
\section{Lemmas for reduction}\label{SectionLemmasForInduction}
%%%%%%%%%%%%%%%%%%%%%%%%%%%%%%%%%%%%%%%%%%%%%%%%%%%%%%%%%%%%%%%%%%%%%%

In this section we assume that $n \geq 2$. To work with the case of $n-1$ we will introduce the following notations.
Denote by $\widetilde{V} = \K^{n-1}$ the natural representation of $\Sym_{n-1}$. For $f \in \K[V^m]$ (resp. $S \subset \K[V^m]$) denote by $\widetilde{f}$ (resp. $\widetilde{S}$) the image of $f$ (resp. $S$) under the surjective $\K$-algebra homomorphism $\K[V^m] \to \K[\widetilde{V}^m]$ which maps $x(j)_n$ to $0$ for $j = 1 \upto m$. Clearly, this homomorphism restricts to a map 
\begin{equation*}\K[V^m]^{\Sym_n} \to \K[\widetilde{V}^m]^{\Sym_{n-1}}, \quad f \mapsto \widetilde{f}.\end{equation*}
 
\noindent{}Furthermore, for $p=(a_1,\ldots,a_m)\in V^m$ define $\widetilde{p}\in \widetilde{V}^m$ by deleting the $n$-th coordinate of every entry $a_j$ of $p$. So the superscript $\,\,\widetilde{}\,\,$ always means ``the $n$-th coordinate is missing''. 

To study the transitions from $n-1$ to $n$ and from $m-1$ to $m$, we will use the following conditions in the formulations of the next lemmas. 

\medskip

%Conditions C3.1---------------------------------------------------------------------
\begin{cond}\label{assump}
Assume $m \geq 2$, $n \geq 2$, $S \subset \K[V^{m-1}]^{\Sym_n}$, $T \subset \K[V^m]^{\Sym_n}$, $p=(a_1 \upto a_m) \in V^m$ and $q = (b_1 \upto b_m) \in V^m$ satisfy
\begin{enumerate}
\item[$\bullet$] $S$ is separating for the $\Sym_n$-action on $V^{m-1}$,
\item[$\bullet$] $\expan{S}{m}\subset T$,
\item[$\bullet$] $\widetilde{T}$ is separating for the $\Sym_{n-1}$-action on $\widetilde{V}^m$,
\item[$\bullet$] $p$ and $q$ can not be separated by $T$.
\end{enumerate}
\end{cond}

\medskip
The next lemma is trivial.

%Lemma L3.2---------------------------------------------------------------------
\begin{lem}\label{LemmaInductionOnN}
Assume that $p=(a_1 \upto a_m) \in V^m$ and $q = (b_1 \upto b_m) \in V^m$ satisfy $(a_1)_n = (b_1)_n \upto (a_m)_n = (b_m)_n$.  Then $p$ and $q$ are in the same $\Sym_n$-orbit if and only if $\widetilde{p}$ and $\widetilde{q}$ are in the same $\Sym_{n-1}$-orbit. 

In particular, if $S \subset \K[V^m]^{\Sym_n}$ is separating, then $\widetilde{S} \subset \K[\widetilde{V}^m]^{\Sym_{n-1}}$ is also separating.
\end{lem}

\medskip
%Lemma L3.3---------------------------------------------------------------------
\begin{lem}\label{LemmaFact0} 
Assume that Conditions~\ref{assump} hold. Then there exists a permutation $\sigma \in \Sym_n$ such that 
\[ \sigma q  = (a_1 \upto a_{m-1},\, c) \quad \text{ (for some } c \in V).\]
\end{lem}
\begin{proof}
Delete the last column vector from $p$ and $q$ to obtain
\[ p' = (a_1 \upto a_{m-1}) \text{ and } q' = (b_1 \upto b_{m-1}) \in V^{m-1}.\]
By the definition of the expansion, for all $f \in S$ we have $f^{(1,2 \upto m-1)} \in \expan{S}{m} \subset T$. Hence the assumption that $p$ and $q$ can not be separated by $T$, implies that for all $f \in S$ we have 
\begin{align*} f(p') = f^{(1,2 \upto m-1)}(p) = f^{(1,2 \upto m-1)}(q) = f(q').\end{align*}
Since $S$ is separating, there exists $\sigma \in \Sym_n$ such that $\sigma q' = p'$. Hence with $c = \sigma b_m$ we obtain
\[\sigma q = (a_1 \upto a_{m-1},\,c).\]
\end{proof}
\medskip

We will often apply Lemma~\ref{LemmaFact0} to replace $q \in V^m$ by $\sigma q$. For this purpose we add the following trivial remark.

\medskip
%Remark 3.4---------------------------------------------------------------------
\begin{remark}\label{RemarkReplacePAndQ}
For $p$, $q \in V^m$, $S \subseteq K[V^m]^{\Sym_n}$ and $\sigma \in \Sym_n$ the statements
\begin{enumerate}
\item[$\bullet$] $p$ and $q$ can not be separated by $S$;
\item[$\bullet$] $p$ and $q$ are in the same $\Sym_n$-orbit;
\end{enumerate} 
do not change their validity when replacing $p$ or $q$ (or both) with $\sigma p$ or $\sigma q$, respectively.
\end{remark}
\medskip

In the next two lemmas we will use the following notation. For an associative binary operation $\ast$ on a set $X$ and $a_1 \upto a_m \in X$ we write 
\[ a_1 \ast \ldots \ast \widehat{a_i}  \ast \ldots \ast \widehat{a_j} \ast \ldots \ast a_m \] 
for 
\[ a_1 \ast \ldots \ast a_{i-1}\ast a_{i+1}  \ast \ldots \ast a_{j-1}\ast a_{j+1}  \ast \ldots \ast a_m\] where $i,\,j \in \{ 1 \upto m\}$. 

\medskip
%Lemma L3.5------------------------------------------------------------------
\begin{lem}\label{LemmaFact2} 
Assume that Conditions~\ref{assump} hold and that there exist $i,\,j \in \{1 \upto m\}$ with $i \neq j$ such that for the stabilizers in $\Sym_n$ we have
\begin{equation}\label{eqLemmaFact2}
G_{a_1} \cap \ldots \cap \widehat{G_{a_i}}\cap \ldots \cap \widehat{G_{a_j}}\cap \ldots \cap G_{a_{m}} \subset G_{a_i}.
\end{equation}
Then $p$ and $q$ are in the same $\Sym_n$-orbit.
\end{lem}
\begin{proof} Without loss of generality we can assume that $i=m-1$ and $j=m$. Furthermore, after applying Lemma~\ref{LemmaFact0} and Remark~\ref{RemarkReplacePAndQ}, we can assume that 
\[p=(a_1 \upto a_{m-1},\,a_m)\quad \text{ and } \quad q=(a_1 \upto a_{m-1},\,c).\]
Consider the reduced vectors
\begin{align*} p':= (a_1 \upto a_{m-2},\,a_m) \quad \text{ and } \quad q' :=  (a_1 \upto  a_{m-2},\,c) \in V^{m-1}.\end{align*}
For all $f \in S$ we have $f^{(1,2 \upto m-2,m)} \in \expan{S}{m} \subset T$ and hence
\begin{align*} f(p') = f^{(1,2 \upto m-2,m)}(p) = f^{(1,2 \upto m-2,m)}(q) =f(q').\end{align*}
Since $S$ is separating, then there exists $\sigma \in \Sym_n$ such that $\sigma q' = p'$. Then $\sigma c = a_m$ and 
$$\sigma \in  G_{a_1} \cap \ldots \cap G_{a_{m-2}}.$$
Thus (\ref{eqLemmaFact2}) implies $\sigma \in G_{a_{m-1}}$ and $\sigma q = p$ follows.
\end{proof}
\medskip

%Lemma L3.6------------------------------------------------------------------
\begin{lem}\label{LemmaFact1} 
Assume that Conditions~\ref{assump} hold and that there exists $l \in \{ 1 \upto n\}$ such that $(a_1)_l = (b_1)_l \upto (a_m)_l = (b_m)_l$.
Then $p$ and $q$ are in the same $\Sym_n$-orbit. 
\end{lem}
\begin{proof} 
By Remark~\ref{RemarkReplacePAndQ} we can apply the transposition $\tau = (l,\,n)$ to both $p$ and $q$ and hence assume that $l=n$.

Denote by $e$ the vector $(1,\ldots,1)^T \in V$. Then $c=(\ga_1 e,\ldots, \ga_m e)\in V^m$ is a fixed point of $\Sym_n$ for any $\ga_1,\ldots,\ga_m\in\K$. Consequently, $p$ and $q$ are in the same $\Sym_n$-orbit if and only if $p - c$ and  $q - c$ are in the same $\Sym_n$-orbit.

Let $c= ((a_1)_n \cdot e, \ldots, (a_m)_n \cdot e)\in V^m$. Replacing $p$ and $q$ by $p-c$ and $q-c$, respectively, we may assume that $0 = (a_j)_n= (b_j)_n$ for all $1\leq j\leq m$. Then  $\widetilde{f}(\widetilde{p})=f(p)=f(q)=\widetilde{f}(\widetilde{q})$ for all $f\in T$. Since $\widetilde{T}$ is separating,  $\widetilde{p}$ and $\widetilde{q}$ are in the same $\Sym_{n-1}$-orbit, and Lemma~\ref{LemmaInductionOnN} concludes the proof.
\end{proof}

\medskip
%Lemma L3.7---------------------------------------------------------------------
\begin{lem}\label{LemmaFact3}
Assume that Conditions~\ref{assump} hold and let $A_i$ denote the partitions $A_i := \parti(a_i)$ for $i \in \{1 \upto m\}$. Moreover, assume that there exist $i,\,j \in \{1 \upto m\}$ with $i\neq j$ such that
\begin{equation}\label{eqLemmaFact3} 
\min(A_1 \sqcap \ldots \sqcap \widehat{A_i} \sqcap \ldots\sqcap \widehat{A_j} \sqcap \ldots \sqcap A_m) = 1.\end{equation}
Then $p$ and $q$ are in the same $\Sym_n$-orbit.
\end{lem}
\begin{proof}
Without loss of generality we can assume that $i=m-1$ and $j=m$.  Furthermore, after applying Lemma~\ref{LemmaFact0} and Remark~\ref{RemarkReplacePAndQ}, we can assume that 
\[p=(a_1 \upto a_{m-1},\,a_m)\quad \text{ and } \quad q=(a_1 \upto a_{m-1},\,c).\]
We consider the reduced vectors 
\begin{align*} p' := (a_1 \upto  a_{m-2},\,a_m) \quad \text{ and } \quad  q' := (a_1 \upto a_{m-2},\,c) \in V^{m-1}\end{align*}
and conclude as in the proof of Lemma~\ref{LemmaFact2} that there exists a permutation $\sigma \in \Sym_n$ such that $\sigma q' = p'$. 
By (\ref{eq1}) and (\ref{StabilizerOfVector}) we have 
\[ \sigma \in G_{a_1} \cap \ldots \cap G_{a_{m-2}} = G_{A_1} \cap \ldots \cap G_{A_{m-2}} = G_{A_1 \sqcap \ldots \sqcap A_{m-2}}.\]
Hence by (\ref{eqLemmaFact3}) there exists a position $l \in \{1 \upto n \}$ such that $\sigma(l) = l$. But since $\sigma c = a_m$, this implies $c_l = (a_m)_l$ and hence the $l$-th coordinates of the entries of $p$ and $q$ are the same. Lemma~\ref{LemmaFact1} concludes the proof.
\end{proof}

%%%%%%%%%%%%%%%%%%%%%%%%%%%%%%%%%%%%%%%%%%%%%%%%%%%%%%%%%%%%%%%%%%%%%%
\section{Proof of Theorem~\ref{TheoremMain}}\label{SectionProofOfTheorem1}
%%%%%%%%%%%%%%%%%%%%%%%%%%%%%%%%%%%%%%%%%%%%%%%%%%%%%%%%%%%%%%%%%%%%%%

\begin{proof_of}{of Theorem~\ref{TheoremMain}}. 
By Remark~\ref{remark_new} we can assume that $S$ is a generating set for $\K[V^{m_0}]^{\Sym_n}$.

We use induction on $n$. For $n = 1$ we have $m_0 \geq 1$ and  we can assume that $S \supset \{x(1)_1\}$. Then $S^{[m]}$ generates the algebra $\K[x(1)_1,\ldots,x(m)_1]$.

Now assume that $n\geq2$. Since $S\subset \K[V^{m_0}]^{\Sym_n}$ is separating, we know by Lemma~\ref{LemmaInductionOnN} that $\widetilde{S}$ is separating for the $\Sym_{n-1}$-action on $\widetilde{V}^{m_0}$. Clearly, \[ m_0 \geq \lfloor \frac{n-1}{2} \rfloor + 1,\] so by induction on $n$ we can use Theorem~\ref{TheoremMain} for $n-1$. It follows that 
\begin{equation}\label{eqSITilde}
\text{$\expan{(\widetilde{S})}{m}$ is separating for the $\Sym_{n-1}$-action on $\widetilde{V}^{m}$.}\end{equation}

Let us prove that $\expan{S}{m}$ is separating by induction on $m \geq m_0$. For $m = m_0$ there is nothing to show. Assume that $m>m_0$ and consider two vectors 
\[ p = (a_1 \upto a_{m-1},\,a_m) \in V^m \quad \text{ and } \quad q \in V^m \]
that can not be separated by $\expan{S}{m}$. To conclude the proof, we have to show that $p$ and $q$ are in the same $\Sym_n$-orbit.  
By induction on $m$ we have that 
\begin{equation}\label{eqSIexpM-1} \text{$\expan{S}{m-1}$ is separating for the $\Sym_n$-action on $V^{m-1}$.}\end{equation}
Since statements (\ref{eqSITilde}), (\ref{eqSIexpM-1}) hold and $\expan{(\expan{S}{m-1})}{m} = \expan{S}{m}$, Conditions~\ref{assump} are satisfied with $\expan{S}{m-1}$ taking the role of $S$ and $\expan{S}{m}$ taking the role of $T$. We apply Lemma~\ref{LemmaFact0} together with Remark~\ref{RemarkReplacePAndQ} and hence can assume that $p$ and $q$ are of the form
\begin{equation*}p = (a_1 \upto a_{m-1},\,a_m),\quad q = (a_1 \upto a_{m-1},\,c) \in V^m.\end{equation*}
Consider the partitions $A_i := \parti(a_i)$ for $i = 1 \upto m-1$ as defined in Section \ref{SectionPartitions}. We split the rest of the proof into two cases.

\medskip
\textbf{\underline{Case 1}}: Assume that there exists an $i \in \{1 \upto m-1\}$ such that 
\begin{equation*}G_{A_1} \cap \ldots \cap G_{A_{i-1}} \subset G_{A_i}.\end{equation*}
Then inclusion (\ref{eqLemmaFact2}) is valid. It follows from Lemma~\ref{LemmaFact2} that $p$ and $q$ are in the same $\Sym_n$-orbit.

\medskip
\textbf{\underline{Case 2}}: Assume that for all $i \in \{1 \upto m-1\}$ we have 
\begin{equation*}G_{A_1} \cap \ldots \cap G_{A_{i-1}} \not\subset G_{A_i}.\end{equation*}
We will only use these assumptions for  $i\leq m -2$. Since \[ m-2 \geq m_0 - 1 \geq \lfloor \frac{n}{2} \rfloor,\] we can apply Lemma \ref{Lemma2OnPartitions} to  $A_1,\ldots,A_{m-2}$ and obtain that
\begin{equation*} \min(A_1 \sqcap \ldots \sqcap A_{m-2}) = 1.\end{equation*}
Then Lemma~\ref{LemmaFact3} concludes the proof.
\end{proof_of}

%%%%%%%%%%%%%%%%%%%%%%%%%%%%%%%%%%%%%%%%%%%%%%%%%%%%%%%%%%%%%%%%%%%%%%
\section{Minimal Separating Sets}\label{MinimalSeparating}
%%%%%%%%%%%%%%%%%%%%%%%%%%%%%%%%%%%%%%%%%%%%%%%%%%%%%%%%%%%%%%%%%%%%%%

This section deals with the question of \emph{minimality} with respect to inclusion of a separating set of $\K[V^m]^{\Sym_n}$ among all separating sets of $\K[V^m]^{\Sym_n}$. At first we show that the method of expansion behaves well in this regard when dealing with certain sets. We call a set of invariants $S \subset \K[V^m]^{\Sym_n}$ {\it elementary} if 
\begin{enumerate}
\item[$\bullet$] any element of $S$ is equal to $\si_t(\un{k})$ for some $1\leq t\leq n$ and $\un{k}\in\N^m$;
\item[$\bullet$] if $\si_t(\un{k})\in S$, then $\si_l(\un{k})\in S$ for all $1\leq l\leq t$.
\end{enumerate}

\medskip
%Lemma L6.1---------------------------------------------------------------------
\begin{lem}\label{LemmaMinSep} Let $m_0,\,m$ be integers with $1 \leq m_0 \leq m$ and let $S \subset \K[V^{m_0}]^{\Sym_n}$ be an elementary set of invariants such that 
\begin{enumerate} 
  \item[(a)] $S$ is a minimal separating set;
  \item[(b)] $\expan{S}{m}$ is a separating set;
  \item[(c)] if $\sigma_t(\un k) \in S$ and $k_i = 0$ for some $i\in\{1,\ldots,m_0\}$, then $\sigma_t(\un r) \in S$ for each $\un{r}\in\N^{m_0}$ from the list: $(0,k_1,\ldots,\widehat{k_i},\ldots,k_{m_0})$, $(k_1,0,k_2,\ldots,\widehat{k_i},\ldots,k_{m_0})$, $\ldots$,
  $(k_1,k_2,\ldots,\widehat{k_i},\ldots,k_{m_0},0)$.
\end{enumerate}
Then $\expan{S}{m}$ is also a minimal separating set.
\end{lem}
\begin{proof} We need to show that for every $F \in \expan{S}{m}$ the set $\expan{S}{m} \setminus \{F\}$ is \emph{not} separating. For such an $F$ by (\ref{DefExpansion}) there exists $f \in S$ and an $m$-admissible $\un j \in \N^{m_0}$ such that $F = f^{(\un j)}$. By assumption, $S \setminus \{f\}$ is not separating. So there exist
\[ p = (a_1 \upto a_{m_0}) \quad \text{ and } \quad q = (b_1 \upto b_{m_0}) \in V^{m_0} \]
with $f(p) \neq f(q)$ and $h(p) = h(q)$ for all $h \in S \setminus \{ f \}$. Consider 
\[ p'=(0,\ldots,0, \underbrace{a_1,}_{\text{position } j_1}0,\;\;\ldots\;\;,0, \underbrace{a_{m_0},}_{\text{position } j_{m_0}} 0, \ldots , 0) \in V^m\]
and $q' \in V^m$ which is defined accordingly. Hence 
$$F(p') = f(p) \neq f(q) = F(q').$$ 
We finish the proof by showing 
\begin{equation*}
H(p') = H(q')\text{ for all } H \in \expan{S}{m} \setminus \{F\}. 
\end{equation*}
By (\ref{DefExpansion}) for any such $H$ there exist $h \in S$ and an $m$-admissible $\un l \in \N^{m_0}$ such that $H = h^{(\un l)}$. Since $S$ is elementary, $h = \si_t(\un k)$ with $t \in \{ 1 \upto n\}$ and $\un k \in \N^{m_0}$. Hence $H = \si_t(\un{k'})$ with 
\[ \un{k'} = (0,\ldots,0, \underbrace{k_1,}_{\text{position } l_1}0,\;\;\ldots\;\;,0, \underbrace{k_{m_0},}_{\text{position } l_{m_0}} 0, \ldots , 0)\in\N^m .\]
If $\un{k'}$ has a non-zero entry at a position different from $j_1 \upto j_{m_0}$, then 
\[ H(p') = 0 = H(q')\]
and we are done. Otherwise, 
$\un{k'}$ is zero at all positions $i \notin \{ j_1 \upto j_{m_0} \} \cap \{l_1 \upto l_{m_0} \}$. Let $\{k_{t_1},\ldots,k_{t_s}\}$ be the set of all non-zero elements of $\un{k}$, where $1\leq t_1<\cdots <t_s\leq m_0$ and $s\geq1$. Since $l_{t_1},\ldots, l_{t_{s}}\in \{ j_1 \upto j_{m_0}\}$, we set $l_{t_1}=j_{d_1},\ldots,l_{t_s}=j_{d_s}$ for some $1\leq d_1<\cdots< d_s\leq m_0$. 
Hence we can view $H$ as $H = \sigma_t(\un r)^{(\un j)}$ for 
\[ \un{r} = (0,\ldots,0, \underbrace{k_{t_1},}_{\text{position } d_1}0,\;\;\ldots\;\;,0, \underbrace{k_{t_s},}_{\text{position } d_s} 0, \ldots , 0)\in\N^{m_0} .\]
Since $\si_t(\un k)\in S$, condition (c) implies that $\sigma_t(\un r) \in S$. The inequality  $F \neq H$ implies $\sigma_t(\un r) \neq f$. It follows that $H(p') = \sigma_t(\un r)(p) = \sigma_t(\un r)(q) = H(q')$ and we are done.
\end{proof}
\medskip

The following example shows that condition (c) from Lemma~\ref{LemmaMinSep} can not be removed.  Namely, in case $\charakt(\K) \neq 2$ and $n=2$ the set 
$$S = \{ \tr(1,0),\; \tr(2,0),\; \tr(0,1),\;\si_2(0,1),\;\tr(1,1) \}\subset \K[V^2]^{\Sym_2 }$$ 
is a minimal separating set and the set $S^{[3]}$ is separating, since we can apply the equality  $2 \si_2(0,1) = \tr(0,1)^2 - \tr(0,2)$ to part (a) of Theorem~\ref{TheoremMainMin} (see below). Since invariants $h_1=\tr(0,1,0)$, $h_2=\tr(0,2,0)$, $h_3=\si_2(0,1,0)$ belong to $S^{[3]}$ and $2 h_3 = h_1^2- h_2$, the set $S^{[3]}$ is not a minimal separating set.

In \cite[Theorem 2]{reimers2019} minimal separating sets for $m = 2$ and $n = 2,\,3,\,4$ were constructed in characteristic 0. 

%Remark 6.2---------------------------------------------------------------------
\begin{remark}\label{RemarkNewExamples}
Theorem~1 and Theorem~2 of \cite{reimers2019} remain true in the case of \linebreak$\charakt(\K) > n$.
\end{remark}

Using our previous results we extend this to $m \geq 3$ in the following theorem. Note in particular, that for $n = 4$ and $m \geq 3$ we can not simply take the expansion of a separating set for $m = 2$.

\medskip
%Theorem 2--------------------------------------------------------------------
\begin{thm}\label{TheoremMainMin} Let $\charakt(\K) = 0$ or $\charakt(\K) > n$. Then for $n=2,\,3,\,4$ and $m \geq 2$ the set $T_{n,m}$ is a minimal separating set for the algebra of invariants $\K[V^{m}]^{\Sym_n}$, where %the sets $T_{n,m}\subset\N^m$ are defined as follows: 
\begin{enumerate}[label=(\alph*)] 
\item for $n=2$ let \begin{align*} &T_{2,2} = \{\tr(r,0),\tr(0,r),\tr(1,1) \mid r=1,2\}, \\ &T_{2,m}=\expan{T_{2,2}}{m} \quad \text{for all } m \geq 3; \end{align*}
\item for $n=3$ let \begin{align*} &T_{3,2}=\{\tr(r,0),\tr(0,r),\tr(1,1),\tr(2,1) \mid r = 1,2,3\}, \\ &T_{3,m}=\expan{T_{3,2}}{m} \quad \text{ for all } m \geq 3;\end{align*}
\item for $n=4$ let  \begin{align*} &T_{4,2}=\{\tr(r,0),\tr(0,r),\tr(1,1),\tr(2,1),\tr(1,2),\tr(3,1)\mid  r =1,2,3,4\},\\ &T_{4,3}=\expan{T_{4,2}}{3} \cup \{\tr(1,1,1)\},\\ &T_{4,m}=\expan{T_{4,3}}{m}  \quad \text{for all $m \geq 4$.}\end{align*}
\end{enumerate}
\end{thm}
\begin{proof} 
The case $m = 2$ follows from \cite[Theorem~2]{reimers2019} together with Remark~\ref{RemarkNewExamples}.

For $n = 2$ and $n = 3$ we can take $m_0 = 2$ in Theorem~\ref{TheoremMain}. %\[ m_0=\lfloor \frac{n}{2} \rfloor +1=2.\] 
Since $T_{n,2}$ is separating, we have that $T_{n,m}$ is separating for all $m \geq 3$ by Theorem~\ref{TheoremMain}. The minimality of $T_{n,m}$ follows from Lemma~\ref{LemmaMinSep}.

Assume $n = 4$. Let us show that
\[ T_{4,3} \text{ is separating.}\] Consider $p = (a_1,\,a_2,\,a_3) \in V^3$ and $q \in V^3$ that can not be separated by $T_{4,3}$. Conditions~\ref{assump} are satisfied for $S := T_{4,2}$ and $T := T_{4,3}$, because $\widetilde{T_{4,3}} \supset T_{3,3}$ is separating. The combination of Lemma~\ref{LemmaFact0} and Remark~\ref{RemarkReplacePAndQ} allows us to assume that
\begin{equation*} p = (a_1,\,a_2,\,a_3), \quad q = (a_1,\,a_2,\,c).\end{equation*}
If for some $i \neq j$ with $i,\,j \in \{1,\,2,\,3\}$ we have $G_{a_i} \subset G_{a_j}$, we can apply Lemma~\ref{LemmaFact2} to obtain that $p$ and $q$ are in the same $\Sym_4$-orbit. Hence we can assume that for all $i \neq j$ we have $G_{a_i} \not\subset G_{a_j}$. Then in particular for all $i \in \{1,\,2,\,3\}$ we have $G_{a_i} \neq \Sym_4$ and $G_{a_i} \neq \{ \id \}$. The former means that not all coordinates of $a_i$ are equal, while the latter means that not all coordinates of $a_i$ are pairwise distinct. Therefore, for the corresponding partition $A_i := \parti(a_i)$ we know that $|A_i| \in \{ 2,\,3\}$ for all $i$. 

If $\min(A_i)=1$ for some $i$, then we can apply Lemma~\ref{LemmaFact3} to conclude that $p$ and $q$ are in the same $\Sym_4$-orbit. Thus the only remaning case is that for all $i \in \{ 1,\,2,\,3\}$ we have $|A_i| = 2$ and $\min(A_i)=2$.
Then for each $i$ the vector $a_i \in \K^4$ has exactly two pairs of distinct coordinates. After applying some permutation to $p$ and $q$ (see Remark~\ref{RemarkReplacePAndQ}) we can assume that 
\[ a_1 = \begin{pmatrix}\lambda_1 \\ \lambda_1 \\ \lambda_2 \\ \lambda_2 \end{pmatrix}, \quad a_2 = \begin{pmatrix}\mu_1\\ \mu_2 \\ \mu_1 \\ \mu_2 \end{pmatrix}, \quad a_3 = \begin{pmatrix}\nu_1\\ \nu_2 \\ \nu_2 \\ \nu_1 \end{pmatrix}\]
with $\lambda_1,\,\lambda_2,\,\mu_1,\,\mu_2,\,\nu_1,\,\nu_2 \in \K$ such that 
\begin{equation*}
\lambda_1\neq \lambda_2,\quad \mu_1\neq \mu_2, \quad \nu_1\neq \nu_2.
\end{equation*}
The equations $\tr(r,0,1)(p)=\tr(r,0,1)(q)$ with $r=0,1$ imply $c_1+c_2=c_3+c_4$. Similarly, $\tr(0,r,1)(p)=\tr(0,r,1)(q)$ with $r=0,1$ imply $c_1+c_3 = c_2+c_4$. Since $\charakt(\K)\neq2$, we obtain $c_3=c_2$ and $c_4=c_1$. If $c = a_3$, then $p = q$ and we are done. So let us assume $c \neq a_3$. Then the equations $\tr(0,0,r)(p) = \tr(0,0,r)(q)$ with $r = 0,1$  imply that
%The four equations $\tr(0,0,r)(p) = \tr(0,0,r)(q)$ with $r = 1 \upto 4$ imply %that there exists a permutation $\sigma \in \Sym_4$ with $\sigma c = a_3$. If %there exists $l \in \{1 \upto 4\}$ with $c_l = (a_3)_l$, then %Lemma~\ref{LemmaFact1} shows that $p$ and $q$ are in the same $\Sym_4$-orbit. %Hence we can assume that $c_l \neq (a_3)_l$ for all $l$. Because of $\sigma c = %a_3$ we have
\[ c = \begin{pmatrix}\nu_2\\ \nu_1 \\ \nu_1 \\ \nu_2 \end{pmatrix}.\]
Since $\tr(1,1,1)$ belongs to $T_{4,3}$, we obtain \begin{align*} 0 &= \tr(1,1,1)(p) - \tr(1,1,1)(q) \\ &= \lambda_1 \mu_1 (\nu_1 - \nu_2) + \lambda_1 \mu_2 (\nu_2 - \nu_1) + \lambda_2 \mu_1 (\nu_2 - \nu_1) + \lambda_2 \mu_2 (\nu_1 - \nu_2) \\
&= (\nu_1 - \nu_2) (\mu_1 - \mu_2) (\lambda_1 - \lambda_2);\end{align*}
a contradiction. This shows that $T_{4,3}$ is separating.
%Now because of (\ref{pqinProofOfTheorem2}) the four equations
%\[ f_{0,0,1}(p) = f_{0,0,1}(q), \quad f_{1,0,1}(p) = f_{1,0,1}(q), \quad f_{0,1,1}(p) = f_{0,1,1}(q), \quad f_{1,1,1}(p) = f_{1,1,1}(q)\]
%translate to 
%\[ \begin{pmatrix}1&1&1&1\\ \lambda_1 &\lambda_1 &\lambda_2&\lambda_2 \\ \mu_1 &\mu_2 &\mu_1&\mu_2\\ \lambda_1 \mu_1& \lambda_1 \mu_2 & \lambda_2 \mu_1 & \lambda_2 \mu_2 \end{pmatrix} \cdot a_3 =  \begin{pmatrix}1&1&1&1\\ \lambda_1 &\lambda_1 &\lambda_2&\lambda_2 \\ \mu_1 &\mu_2 &\mu_1&\mu_2\\ \lambda_1 \mu_1& \lambda_1 \mu_2 & \lambda_2 \mu_1 & \lambda_2 \mu_2 \end{pmatrix} \cdot c.\]
%It is easy to check that this matrix invertible, hence $a_3 = c$ and $p = q$ follows. This shows that $S(I_{4,3})$ is separating.

Now let us show that $T_{4,3}$ is a minimal separating set. For $f \in T_{4,3}$ with $f \neq \tr(1,1,1)$ the minimality of $T_{4,2}$ implies that $T_{4,3} \setminus \{f \}$ is not separating. For $f = \tr(1,1,1)$ consider the vectors $p=(a_1,\,a_2,\,a_3)$ and $q=(a_1,\,a_2,\,c) \in V^3$ with
\[ a_1 = \begin{pmatrix}1\\1\\2\\2\end{pmatrix}, \quad a_2= \begin{pmatrix}1\\2\\1\\2\end{pmatrix},\quad a_3= \begin{pmatrix}1\\2\\2\\1\end{pmatrix},\quad c= \begin{pmatrix}2\\1\\1\\2\end{pmatrix}.\]
Here $\tr(1,1,1)(p) \neq \tr(1,1,1)(q)$, but $p$ and $q$ can not be separated by $T_{4,3} \setminus \{ \tr(1,1,1) \}$.
So the theorem is proven for $n = 4$ and $m = 3$. 

Since for $n = 4$ we can take $m_0 = 3$ in Theorem~\ref{TheoremMain}, the set $T_{4,m}$ is separating for all $m \geq 4$. The minimality of $T_{4,m}$ follows from Lemma~\ref{LemmaMinSep}.
\end{proof}
\medskip

Now we assume that $\charakt(\K) =  0$ or $\charakt(\K)> n$ and compare the asymptotic sizes of multi-homogeneous generating and separating sets when $n$ is fixed and $m$ tends to infinity. Denote by $M_m$ the minimal generating set of $\K[V^m]^{\mathcal{S}_n}$ from (\ref{eqMinimalGen}) depending on $m$. Theorem~\ref{TheoremMain} allows us to use $M_{m_0}$, where $m_0 := \lfloor \frac n 2 \rfloor + 1$, to construct separating sets of $\K[V^m]^{\Sym_n}$ for $m > m_0$. We define
\[ S_m := \begin{cases}M_m \quad &\text{ for } m \leq m_0, \\[1mm] \expan{M_{m_0}}{m} \quad &\text{ for } m> m_0. \end{cases} \]
By Theorem~\ref{TheoremMain}, $S_m$ is a separating set for all $m$. 
For two functions $f,\,h:\N\to \R_{>0}$ we write $f\sim h$ if $\lim_{m\to\infty}\frac{f(m)}{h(m)}=1$.

\medskip
%Remark Rem5.1---------------------------------------------------------------------
\begin{remark}\label{RemarkFraction}
Let $\charakt(\K) =  0$ or $\charakt(\K)> n$. Then as functions of $m$ 
\[ \frac{|S_m|}{|M_m|}\sim \binom{n}{m_0} \frac{n!}{m_0!} \; \frac{1}{m^{n-m_0}}. \] In particular for $n \geq 3$, we have \[ \lim_{m \to \infty} \frac{|S_m|}{|M_m|} = 0.\]
\end{remark}

\bigskip

%-------------------------------
%Bibliography
%-------------------------------

%\bibliographystyle{myplain}
\bibliographystyle{siam}
\bibliography{literature}

\end{document}